\documentclass[12pt]{amsart}
\usepackage[latin1]{inputenc}
\usepackage{mathptmx}
\usepackage{amscd}
\usepackage{amssymb}
\newtheorem{theorem}{Theorem}
\newtheorem{lemma}[theorem]{Lemma}
\newtheorem{corollary}[theorem]{Corollary}

\theoremstyle{definition}
\newtheorem{remark}[theorem]{Remark}
\newtheorem{definition}[theorem]{Definition}

\let\phi=\varphi

\def\card{\operatorname{card}}

\def\Ann{\operatorname{Ann}}
\def\Ass{\operatorname{Ass}}
\def\Min{\operatorname{Min}}

\let\oldbigwedge\bigwedge
\def\BIGwedge{{\textstyle\oldbigwedge}}
\def\medwedge{{\scriptstyle\oldbigwedge}}
\def\bigwedge{\mathchoice{\BIGwedge}{\BIGwedge}{\medwedge}{}}

\DeclareMathOperator{\Nil}{Nil}
\DeclareMathOperator{\rad}{rad}
\DeclareMathOperator{\FId}{FId}
\DeclareMathOperator{\Max}{Max}

\DeclareMathOperator{\zd}{zd}
\DeclareMathOperator{\diam}{diam}

\let\epsilon=\varepsilon

\begin{document}
\title{Zero-Divisors of Content Algebras}

\author{Peyman Nasehpour}
\address{Universit\"at Osnabr\"uck, FB Mathematik/Informatik, 49069
Osnabr\"uck, Germany} \email{Peyman.Nasehpour@mathematik.uni-osnabrueck.de} \email{Nasehpour@gmail.com}

\begin{abstract}
In this article, we prove that in content extentions minimal primes extend to minimal primes and discuss zero-divisors of a content algebra over a ring who has Property (A) or whose set of zero-divisors is a finite union of prime ideals. We also examine the preservation of diameter of zero-divisor graph under content extensions.
\end{abstract}

\maketitle

\section{Introduction}

Throughout this paper all rings are commutative with unit and all modules are assumed to be unitary\footnote{2000 Mathematics Subject Classification: 13A15, 13B25, 05C99.
Keywords: Content Algebra, Few Zero-divisors, McCoy's Property, Minimal Prime, Property (A), Primal Ring, Zero-divisor Graph}. In this paper, we discuss zero-divisors of content algebras. To this end, one needs to know about content modules and algebras introduced in [OR]. Our main goal is to show that many results of the zero-divisors of polynomial rings are correct for content algebras.

First we recall the essential definitions. Let $R$ be a commutative ring with identity, and $M$ a unitary $R$-module. The \textit{content function}, $c$ from $M$ to the ideals of $R$ is defined by
$$
c(x) = \bigcap \lbrace I \colon I \text{~is an ideal of~} R \text{~and~} x \in IM \rbrace.
$$

$M$ is called a \textit{content $R$-module} if $x \in c(x)M $, for all $x \in M$.

Note that $c(x)$ is a finitely generated ideal of $R$ for all $x\in M$, if $M$ is a content $R$-module [OR, 1.2]. So when $M$ is a content $R$-module, the function $c$ is from $M$ to $\FId(R)$, where by $\FId(R)$, we mean the set of finitely generated ideals of $R$.

Let $R^\prime$ be an $R$-algebra. $R^\prime$ is defined to be a \textit{content $R$-algebra}, if the following conditions hold:

\begin{enumerate}
 \item
$R^\prime$ is a content $R$-module.
 \item
(\textit{Faithful flatness}) For any $r \in R$ and $f \in R^\prime$, the equation $c(rf) = rc(f)$ holds, and $c(R^\prime) = R$.
 \item
(\textit{Dedekind-Mertens content formula}) For all $f$ and $g$ in $R^\prime$, there exists a natural number $n$ such that $c(f)^n c(g) = c(f)^{n-1} c(fg)$.
\end{enumerate}

In section 2, we discuss content and weak content algebras and prove that if $R$ is a ring and $S$, a commutative monoid, then the monoid ring $B=R[S]$ is a content $R$-algebra if and only if one of the following conditions satisfies:

\begin{enumerate}
\item For $f,g \in B$, if $c(f) = c(g) = R$, then $c(fg) = R$.
\item (\textit{McCoy's Property}) For $g \in B$, $g$ is a zero-divisor of $B$ iff there exists $r \in R-\lbrace 0 \rbrace$ such that $rg = 0$.
\item $S$ is a cancellative and torsion-free monoid.
\end{enumerate}

In Section 3, we discuss prime ideals of content and weak content algebras and we show that in content extensions, minimal primes extend to minimal primes. More precisely, if $B$ is a content $R$-algebra, then there is a correspondence between $\Min(R)$ and $\Min(B)$, with the function $ \phi : \Min(R) \longrightarrow \Min(B)$ defined by $\textbf{p} \longrightarrow \textbf{p}B$.

In Section 4, we introduce a family of rings which have very few zero-divisors. It is a well-known result that the set of zero-divisors of a Noetherian ring is a finite union of its associated primes [K, p. 55]. Rings having few zero-divisors have been introduced in [Dav]. We define that a ring $R$ has \textit{very few zero-divisors}, if $Z(R)$ is a finite union of prime ideals in $\Ass(R)$. In this section, we prove that if $R$ is a ring that has very few zero-divisors and $B$ is a content $R$-algebra, then $B$ has very few zero-divisors too.

Another celebrated property of Noetherian rings is that every ideal entirely contained in the set of their zero-divisors has a nonzero annihilator. A ring $R$ has \textit{Property (A)}, if each finitely generated ideal $I \subseteq Z(R)$ has a nonzero annihilator [HK]. In Section 4, we also prove some results for content algebras over rings having Property (A) and then we discuss rings having few zero-divisors in more details. Let us recall that a ring $R$ is said to have few zero-divisors, if the set $Z(R)$ of zero-divisors is a finite union of prime ideals. It is well-known that a ring $R$ has few zero-divisors iff its classical quotient ring $T(R)$ is semi-local [Dav]. We may suppose that $Z(R) = \bigcup _{i=1}^n \textbf{p}_i$ such that $\textbf{p}_i \nsubseteq \bigcup\cup _{j=1 \wedge j \neq i}^n \textbf{p}_j$ for all $ 1\leq i \leq n$. Then we have $\textbf{p}_i \nsubseteq \textbf{p}_j$ for all $i \neq j$ and by Prime Avoidance Theorem, these prime ideals are uniquely determined. In such a case, it is easy to see that $\Max(T(R))= \{\textbf{p}_1T(R), \dots , \textbf{p}_nT(R) \}$, where by $T(R)$ we mean total quotient ring of $R$. Such prime ideals are called maximal primes in $Z(R)$. We denote the number of maximal primes in $Z(R)$ by $\zd(R)$. As one of the main results of this section, we show that if $R$ has Property (A) and $\zd(R)=n$ and $B$ is a content $R$-algebra, then $\zd(B)=n$. At the end of this section, we consider the interesting case, when $\zd(R)=1$, i.e. $Z(R)$ is an ideal of $R$. Such a ring is called a primal ring [Dau].

We let $Z(R)^*$ denote the (nonempty) set of proper zero-divisors of $R$, where by a proper zero-divisor we mean a zero-divisor different from zero. We consider the graph $\Gamma(R)$, called the zero-divisor graph of $R$, whose vertices are the elements of $Z(R)^*$ and edges are those pairs of distinct proper zero-divisors $\{a,b\}$ such that $ab=0$. The last section is devoted to examine the preservation of diameter of zero-divisor graph under content extensions.

Unless otherwise stated, our notation and terminology will follow as closely as possible that of Gilmer [G1]. Note that \textit{iff} always stands for if and only if.

\section{Content Algebras}

Content modules and algebras were introduced in [OR]. Content algebras are actually a natural generalization of polynomial rings [ES]. Let $R$ be a commutative ring with identity. For $f \in R[X]$, the content of $f$, denoted by $c(f)$, is defined as the $R$-ideal generated by the coefficients of $f$. One can easily check that $c(fg) \subseteq c(f)c(g)$ for the two polynomials $f, g \in R[X]$ and may ask when the equation $c(fg) = c(f)c(g)$ holds. Tsang, a student of Kaplansky, proved that if $D$ is an integral domain and $c(f)$, for $f \in D[X]$, is an invertible ideal of $D$, then $c(fg) = c(f)c(g)$, for all $g \in D[X]$. Tsang's guess in [T] was that the converse was true and the correctness of her guess was completely proved some decades later [LR]. Though the equation $c(fg) = c(f)c(g)$ is not always true, a weaker formula always holds that is called the \textit{Dedekind-Mertens content formula} [AG].

\begin{theorem}
 \textbf{Dedekind-Mertens Lemma}\textbf{.} Let $R$ be a ring. For each $f$ and $g$ in $R[X]$, there exists a natural number $n$ such that $c(f)^n c(g) = c(f)^{n-1} c(fg)$.
\end{theorem}

Good examples of content $R$-algebras are the polynomial ring $R[X]$ and the group ring $R[G]$, where $G$ is a torsion-free abelian group [N]. These are actually free $R$-modules. For some examples of content $R$-algebras that as $R$-modules are not free, one can refer to [OR, Examples 6.3, p. 64]. Rush generalized content algebras and defined weak content algebras as follows [R, p. 330]:

\begin{definition}
 Let $R$ be a commutative ring with identity and $R^\prime$ an $R$-algebra. $R^\prime$ is defined to be a \textit{weak content $R$-algebra}, if the following conditions hold:

\begin{enumerate}
 \item
$R^\prime$ is a content $R$-module.
 \item
(\textit{Weak content formula}) For all $f$ and $g$ in $R^\prime$, $c(f)c(g) \subseteq \rad (c(fg))$ (Here $\rad(A)$ denotes the radical of the ideal $A$).
\end{enumerate}

\end{definition}

It is obvious that content algebras are weak content algebras, but the converse is not true. For example if $R$ is a Noetherian ring, then $R[[X_1,X_2,\ldots,X_n]]$ is a weak content $R$-algebra, while it is not a content $R$-algebra [R, p. 331]. We end our introductory section with the following result:

\begin{theorem}
 Let $R$ be a ring and $S$ be a commutative monoid. Then the following statements about the monoid algebra $B = R[S]$ are equivalent:

\begin{enumerate}
 \item $B$ is a content $R$-algebra.
 \item $B$ is a weak content $R$-algebra.
 \item For $f,g \in B$, if $c(f) = c(g) = R$, then $c(fg) = R$.
 \item (\textit{McCoy's Property}) For $g \in B$, $g$ is a zero-divisor of $B$ iff there exists $r \in R-\lbrace 0 \rbrace$ such that $rg = 0$.
 \item $S$ is a cancellative and torsion-free monoid.
 \end{enumerate}

\end{theorem}

\begin{proof}
 $(1) \rightarrow (2) \rightarrow (3)$ and $(1) \rightarrow (4)$ are obvious ([OR] and [R]). Also, according to [N] (5) implies (1). Therefore the proof will be complete if we prove that (3) as well as (4) implies (5).

$(3) \rightarrow (5)$: We prove that if $S$ is not cancellative or not torsion-free then (3) cannot hold. For the moment, suppose that $S$ is not cancellative, so there exist $s,t,u \in S$ such that $s+t = s+u$ while $t \not= u$. Put $f = X^s$ and $g = (X^{t}-X^u)$. Then obviously $c(f) = c(g) = R$, while $c(fg) = (0)$. Finally suppose that $S$ is cancellative but not torsion-free. Let $s,t \in S$ be such that $s \not=t$, while $ns = nt$ for some natural $n$. Choose the natural number $k$ minimal so that $ks = kt$. Then we have $0 = X^{ks}-X^{kt} = (X^s-X^t)(\sum_{i=0}^{k-1} X^{(k-i-1)s+it}) $.

Since $S$ is cancellative, the choice of $k$ implies that $ (k-i_1-1)s+i_{1}t \not= (k-i_2-1)s+i_{2}t $ for $0 \leq i_1 < i_2 \leq k-1 $.
Therefore $\sum_{i=0}^{k-1} X^{(k-i-1)s+it} \not= 0$, and this completes the proof. In a similar way one can prove $(4) \rightarrow (5)$ [G2. p.82].
\end{proof}

\section{Prime Ideals in Content Algebras}

Let $B$ be a weak content $R$-algebra such that for all $\textbf{m} \in \Max(R)$ (by $\Max(R)$, we mean the maximal ideals of $R$), we have $\textbf{m}B \not = B$, then by [R, Theorem 1.2, p. 330], prime ideals extend to prime ideals. Particularly in content algebras primes extend to primes. We recall that when $B$ is a content $R$-algebra, then $g$ is a zero-divisor of $B$, iff there exists an $r \in R- \lbrace 0 \rbrace $ such that $rg = 0$ [OR, 6.1, p. 63]. Now we give the following theorem about associated prime ideals. We assert that by $\Ass_R (M)$, we mean the associated prime ideals of the $R$-module $M$.

\begin{theorem}
 Let $B$ be a content $R$-algebra and $M$ a nonzero $R$-module. If $\textbf{p} \in \Ass_R (M)$ then $\textbf{p}B \in \Ass_B (M \otimes_R B)$.
\end{theorem}

\begin{proof}
 Let $\textbf{p} \in \Ass_R (M)$ and $\textbf{p}=\Ann(x)$, where $x\in M$. Therefore $ 0 \longrightarrow R/\textbf{p} \longrightarrow M$ is an $R$-exact sequence. Since $B$ is a faithfully flat $R$-module, we have the following $B$-exact sequence:
$$ 0 \longrightarrow B/\textbf{p}B \longrightarrow M \otimes_R B$$

with $\textbf{p}B = \Ann (x \otimes_R 1_B)$. Since $B$ is a content $R$-algebra, $\textbf{p}B$ is a prime ideal of $B$.
\end{proof}

Now we give a general theorem on minimal prime ideals in algebras. One of the results of this theorem is that in faithfully flat weak content algebras (including content algebras), minimal primes extend to minimal primes and, more precisely, there is actually a correspondence between the minimal primes of the ring and their extensions in the algebra.

\begin{theorem}
 Let $B$ be an $R$-algebra with the following properties:

\begin{enumerate}
 \item For each prime ideal $\textbf{p}$ of $R$, the extended ideal $\textbf{p}B$ of $B$ is prime.
 \item For each prime ideal  $\textbf{p}$ of $R$, $\textbf{p}B \cap R = \textbf{p}$.
\end{enumerate}

Then the function $ \phi : \Min(R) \longrightarrow \Min(B)$ given by $\textbf{p} \longrightarrow \textbf{p}B$ is a bijection.

\begin{proof}
First we prove that if $\textbf{p}$ is a minimal prime ideal of $R$, then $\textbf{p}B$ is also a minimal prime ideal of $B$. Let $Q$ be a prime ideal of $B$ such that $Q \subseteq \textbf{p}B$. So $Q \cap R \subseteq \textbf{p}B \cap R = \textbf{p}$. Since $\textbf{p}$ is a minimal prime ideal of $R$, we have $Q \cap R =\textbf{p}$ and therefore $ Q = \textbf{p}B $. This means that $ \phi $ is a well-defined function. Obviously the second condition causes $ \phi $ to be one-to-one. The next step is to prove that $ \phi $ is onto. For showing this, consider $Q \in \Min(B)$, so $Q \cap R$ is a prime ideal of $R$ such that $(Q \cap R)B \subseteq Q$ and therefore $(Q \cap R)B = Q$. Our claim is that $(Q \cap R)$ is a minimal prime ideal of $R$. Suppose $\textbf{p}$ is a prime ideal of $R$ such that $ \textbf{p} \subseteq Q \cap R$, then $\textbf{p}B \subseteq Q$ and since $Q$ is a minimal prime ideal of $B$, $\textbf{p}B = Q = (Q \cap R)B$ and therefore $\textbf{p} = Q \cap R$.
\end{proof}

\end{theorem}

\begin{corollary}
 Let $B$ be a weak content and faithfully flat $R$-algebra. Then the function $ \phi : \Min(R) \longrightarrow \Min(B)$ given by $\textbf{p} \longrightarrow \textbf{p}B$ is a bijection.
\end{corollary}

\begin{proof}
Since $B$ is a weak content and faithfully flat $R$-algebra, then for each prime ideal $\textbf{p}$ of $R$, the extended ideal $\textbf{p}B$ of $B$ is prime and also $c(1_B)=R$ by [OR, Corollary 1.6] and [R, Theorem 1.2]. Now consider $r \in R$, then $c(r) = c(r\cdot 1_B) = r\cdot c(1_B) = (r)$. Therefore if $r \in IB \cap R$, then $(r) = c(r) \subseteq I$.  Thus for each prime ideal  $\textbf{p}$ of $R$, $\textbf{p}B \cap R = \textbf{p}$.
\end{proof}

\begin{corollary}
 Let $R$ be a Noetherian ring. Then $ \phi : \Min(R) \longrightarrow \Min(R[[X_1,\ldots,X_n]])$ given by $\textbf{p} \longrightarrow \textbf{p}\cdot (R[[X_1,\ldots,X_n]])$ is a bijection.
\end{corollary}

\section{Content algebras over rings having few zero-divisors}

For a ring $R$, by $Z(R)$, we mean the set of zero-divisors of $R$. In [Dav], it has been defined that a ring $R$ has \textit{few zero-divisors}, if $Z(R)$ is a finite union of prime ideals. We present the following definition to prove some other theorems related to content algebras.

 \begin{definition}
  A ring $R$ has \textit{very few zero-divisors}, if $Z(R)$ is a finite union of prime ideals in $\Ass(R)$.
 \end{definition}

\begin{theorem}
 Let $R$ be a ring that has very few zero-divisors. If $B$ is a content $R$-algebra, then $B$ has very few zero-divisors too.
\end{theorem}

\begin{proof}
 Let $Z(R) = \textbf{p}_1\cup \textbf{p}_2\cup \cdots \cup \textbf{p}_n$, where $\textbf{p}_i \in \Ass_R(R)$ for all $1 \leq i \leq n$. We will show that $Z(B) = \textbf{p}_1B\cup \textbf{p}_2B\cup \cdots \cup \textbf{p}_nB$. Let $g \in Z(B)$, so there exists an $r \in R- \lbrace 0 \rbrace $ such that $rg = 0$ and so $rc(g) = (0)$. Therefore $c(g) \subseteq Z(R)$ and according to Prime Avoidance Theorem, we have $c(g) \subseteq \textbf{p}_i$, for some $1 \leq i \leq n$ and therefore $g \in \textbf{p}_iB$. Now let $g \in \textbf{p}_1B\cup \textbf{p}_2B\cup \cdots \cup \textbf{p}_nB$ so there exists an $i$ such that $g \in \textbf{p}_iB$, so $c(g) \subseteq \textbf{p}_i$ and $c(g)$ has a nonzero annihilator and this means that g is a zero-divisor of $B$. Note that $\textbf{p}_iB \in \Ass_B(B)$, for all $1 \leq i \leq n$.
\end{proof}

\begin{remark}
Let $R$ be a ring and consider the following three conditions on $R$:

\begin{enumerate}
 \item $R$ is a Noetherian ring.
 \item $R$ has very few zero-divisors.
 \item $R$ has few zero-divisors.
\end{enumerate}

Then, $(1) \rightarrow (2) \rightarrow (3)$ and none of the implications is reversible.
\end{remark}

\begin{proof}
For  $(1) \rightarrow(2) $ use [K, p. 55]. It is obvious that $(2) \rightarrow (3)$.

Suppose $k$ is a field, $A=k[X_1, X_2, X_3,\ldots,X_n,\ldots]$ and $\textbf{m} =(X_1, X_2, X_3,\ldots, X_n,\ldots)$ and at last $\textbf{a}=(X_1^2, X_2^2, X_3^2,\ldots, X_n^2,\ldots)$. Since $A$ is a content $k$-algebra and $k$ has very few zero-divisors, $A$ has very few zero-divisors while it is not a Noetherian ring. Also consider the ring $R=A/\textbf{a}$. It is easy to check that $R$ is a quasi-local ring with the only prime ideal $\textbf{m}/\textbf{a}$ and $Z(R)=\textbf{m}/\textbf{a}$ and finally $\textbf{m}/\textbf{a}\notin \Ass_R(R)$. Note that $\Ass_R(R)=\emptyset$.
\end{proof}

Now we bring the following definition from [HK] and prove some other results for content algebras.

\begin{definition}
  A ring $R$ has \textit{Property (A)}, if each finitely generated ideal $I \subseteq Z(R)$ has a nonzero annihilator.
 \end{definition}

Let $R$ be a ring. If $R$ has very few zero-divisors (for example if $R$ is Noetherian), then $R$ has Property (A) [K, Theorem 82, p. 56], but there are some non-Noetherian rings which do not have Property (A) [K, Exercise 7, p. 63]. The class of non-Noetherian rings having Property (A) is quite large [H, p. 2].

\begin{theorem}
 Let $B$ be a content $R$-algebra such that $R$ has Property (A). Then $T(B)$ is a content $T(R)$-algebra, where by $T(R)$, we mean total quotient ring of $R$.
\end{theorem}

\begin{proof}
 Let $S^ \prime = B-Z(B)$. If $S = S^ \prime \cap R$, then $S = R-Z(R)$. We prove that if $c(f) \cap S = \emptyset $, then $f \not\in S^ \prime$. In fact when $c(f) \cap S = \emptyset $, then $c(f) \subseteq Z(R)$ and since $R$ has Property (A), $c(f)$ has a nonzero annihilator. This means that $f$ is a zero-divisor of $B$ and according to [OR, Theorem 6.2, p. 64] the proof is complete.
\end{proof}

\begin{theorem}
 Let $B$ be a content $R$-algebra such that the content function $c: B \longrightarrow \FId(R)$ is onto, where by $\FId(R)$, we mean the set of finitely generated ideals of $R$. The following statements are equivalent:

\begin{enumerate}
 \item $R$ has Property (A).
 \item For all $f \in B$, $f$ is a regular element of $B$ iff $c(f)$ is a regular ideal of $R$.
\end{enumerate}

\begin{proof}
 $(1) \rightarrow (2)$: Let $R$ have Property (A). If $f \in B$ is regular, then for all nonzero $r \in R$, $rf \not= 0$ and so for all nonzero $r \in R$, $rc(f) \not= (0)$, i.e. $\text{Ann}(c(f)) = (0)$ and according to the definition of Property (A), $c(f) \not\subseteq Z(R)$. This means that $c(f)$ is a regular ideal of $R$. Now let $c(f)$ be a regular ideal of $R$, so $c(f) \not\subseteq Z(R)$ and therefore $\text{Ann}(c(f)) = (0)$. This means that for all nonzero $r \in R$, $rc(f) \not= (0)$, hence for all nonzero $r \in R$, $rf \not= 0$. Since $B$ is a content $R$-algebra, $f$ is not a zero-divisor of $B$.

$(2) \rightarrow (1)$: Let $I$ be a finitely generated ideal of $R$ such that $I \subseteq Z(R)$. Since the content function $c: B \longrightarrow \FId(R)$ is onto, there exists an $f \in B$ such that $c(f) = I$. But $c(f)$ is not a regular ideal of $R$, therefore according to our assumption, $f$ is not a regular element of $B$. Since $B$ is a content $R$-algebra, there exists a nonzero $r \in R$ such that $rf = 0$ and this means that $rI = (0)$, i.e. $I$ has a nonzero annihilator.
\end{proof}

\begin{remark}
 In the above theorem the surjectivity condition for the content function $c$ is necessary, because obviously $R$ is a content $R$-algebra and the condition (2) is satisfied, while one can choose the ring $R$ such that it does not have Property (A) [K, Exercise 7, p. 63].
\end{remark}

\end{theorem}

\begin{theorem}
 Let $R$ have property (A) and $B$ be a content $R$-algebra. Then $Z(B)$ is a finite union of prime ideals in $\Min(B)$ iff $Z(R)$ is a finite union of prime ideals in $\Min(R)$.
\end{theorem}

\begin{proof}
 The proof is similar to the proof of Theorem 9 by considering Theorem 5.
\end{proof}

Please note that if $R$ is a Noetherian reduced ring, then $Z(R)$ is a finite union of prime ideals in $\Min(R)$ (Refer to [K, Theorem 88, p. 59] and [H, Corollary 2.4]).

It is well-known that a ring $R$ has few zero-divisors iff its classical quotient ring $T(R)$ is semi-local [Dav]. We may suppose that $Z(R) = \bigcup _{i=1}^n \textbf{p}_i$ such that $\textbf{p}_i \nsubseteq \bigcup _{j=1 \wedge j \neq i}^n \textbf{p}_j$ for all $ 1\leq i \leq n$. Then we have $\textbf{p}_i \nsubseteq \textbf{p}_j$ for all $i \neq j$ and by Prime Avoidance Theorem, these prime ideals are uniquely determined. In such a case, it is easy to see that $\Max(T(R))= \{\textbf{p}_1T(R), \dots , \textbf{p}_nT(R) \}$, where by $T(R)$ we mean total quotient ring of $R$. This is the base for the following definition.

\begin{definition}
 A ring $R$ is said to have \textit{few zero-divisors of degree} $n$, if $R$ has few zero-divisors and $n = \card \Max(T(R))$. In such a case, we write $\zd(R) = n$.
\end{definition}

\begin{remark}
If $R_i$ is a ring having few zero-divisors of degree $k_i$ for all $1\leq i \leq n$, then $\zd(R_1\times \dots \times R_n) = \zd(R_1)+ \dots + \zd(R_n)$.
\end{remark}

Now we give the following theorem:

\begin{theorem}
 Let $B$ be a content $R$-algebra. Then the following statements hold for all natural numbers $n$:

\begin{enumerate}
 \item If $\zd(B) = n$, then $\zd(R) \leq n$.
 \item If $R$ has Property (A) and $\zd(R) = n$, then $\zd(B) = n$.
 \item If the content function $c: B \longrightarrow \FId(R)$ is onto, then $\zd(B) = n$ iff $\zd(R) = n$ and $R$ has Property (A).
\end{enumerate}

\end{theorem}

\begin{proof}
 $(1)$: Let $Z(B) = \bigcup _{i=1}^n Q_i$. We prove that $Z(R) = \bigcup _{i=1}^n (Q_i \cap R)$. In order to do that let $r\in Z(R)$. Since $Z(R) \subseteq Z(B)$, there exists an $i$ such that $r\in Q_i$ and therefore $r\in Q_i \cap R$. Now let $r\in Q_i \cap R$ for some $i$, then $r\in Z(B)$, and this means that there exists a nonzero $g\in B$ such that $rg = 0$ and at last $rc(g) = 0$. Choose a nonzero $d\in c(g)$ and we have $rd=0$.

$(2)$: Note that similar to the proof of Theorem 9, if $Z(R)= \bigcup _{i=1}^n \textbf{p}_i$, then $Z(B)= \bigcup _{i=1}^n \textbf{p}_iB$. Also it is obvious that $\textbf{p}_iB \subseteq \textbf{p}_jB$ iff $\textbf{p}_i \subseteq \textbf{p}_j$ for all $1 \leq i,j \leq n$. These two imply that $\zd(B) = n$.

$3$: $(\leftarrow)$ is nothing but $(2)$. For proving $(\rightarrow)$, consider that by $(1)$, we have $\zd(R) \leq n$. Now we prove that ring $R$ has Property (A). Let $I \subseteq Z(R)$ be a finite ideal of $R$. Choose $f\in B$ such that $I=c(f)$. So $c(f) \subseteq Z(R)$ and by Prime Avoidance Theorem and $(1)$, there exists an $1 \leq i \leq n$ such that $c(f) \subseteq Q_i \cap R$. Therefore $f \in (Q_i \cap R)B$. But $(Q_i \cap R)B \subseteq Q_i$. So $f \in Z(B)$ and according to McCoy's property for content algebras, there exists a nonzero $r\in R$ such that $f.r=0$. This means that $I.r=0$ and $I$ has a nonzero annihilator. Now by $(2)$, we have $\zd(R) = n$.
\end{proof}

\begin{corollary}
 Let $R$ be a ring and $S$ a commutative, cancellative, torsion-free monoid. Then for all natural numbers $n$, $\zd(R[S]) = n$ iff $\zd(R)=n$ and $R$ has Property (A) .
\end{corollary}

\begin{definition}
 An element $r$ of a ring $R$ is said to be \textit{prime to an ideal} $I$ of $R$ if $I : (r) =I$, where by $I : (r)$, we mean the set of all members $c$ of $R$ such that $cr \in I$ ([ZS, p. 223]).
\end{definition}

\begin{theorem}
 Let $R$ be a ring, $I$ an ideal of $R$ and $B$ a content $R$-algebra. Then $f \in B$ is not prime to $IB$ iff $f.r \in IB$ for some $r \in R-I$.
\end{theorem}

\begin{proof}
 If $I$ is an ideal of $R$ and $B$ a content $R$-algebra, then $B/IB$ is a content $(R/I)$-algebra. Assume that $f \in B$ is not prime to $IB$, so there exists $g \in B$ such that $fg \in IB$, while $g \notin IB$. This means that $f+IB$ is a zero-divisor of $B/IB$ and according to McCoy's property, we have $(f+IB)(r+IB)=IB$ for some $r \in R-I$.
\end{proof}

Let $I$ be an ideal of $R$. We denote the set of all elements of $R$ that are not prime to $I$ by $S(I)$. It is obvious that $r\in S(I)$ iff $r+I$ is a zero-divisor of the quotient ring $R/I$. The ideal $I$ is said to be primal if $S(I)$ forms an ideal and in such a case, $S(I)$ is a prime ideal of $R$. A ring $R$ is said to be \textit{primal}, if the zero ideal of $R$ is primal [Dau]. It is obvious that $R$ is primal iff $Z(R)$ is an ideal of $R$. It is easy to check that if $Z(R)$ is an ideal of $R$, it is a prime ideal and therefore $R$ is primal iff $R$ has few zero-divisors of degree one, i.e. $\zd(R)=1$.

\begin{theorem}
 Let $B$ be a content $R$-algebra. Then the following statements hold:

\begin{enumerate}
 \item If $B$ is primal, then $R$ is primal and $Z(B) = Z(R)B$.
 \item If $R$ is primal and has Property (A), then $B$ is primal, has Property (A) and $Z(B) = Z(R)B$.
 \item If the content function $c: B \longrightarrow \FId(R)$ is onto, then $B$ is primal iff $R$ is primal and has Property (A).
\end{enumerate}

\end{theorem}

\begin{proof}
$(1)$: Assume that $Z(B)$ is an ideal of $B$. We show that $Z(R)$ is an ideal of $R$. For doing that it is enough to show that if $a,b \in Z(R)$, then $a+b \in Z(R)$. Let $a,b \in Z(R)$. Since $Z(R) \subseteq Z(B)$ and $Z(B)$ is an ideal of $B$, we have $a+b \in Z(B)$. This means that there exists a nonzero $g \in B$ such that $(a+b)g = 0$. Since $g\neq0$, we can choose $0 \neq d \in c(g)$ and we have $(a+b)d = 0$. Now it is easy to check that $Z(B) = Z(R)B$.

$(2)$: Let $R$ have Property (A) and $Z(R)$ be an ideal of $R$. We show that $Z(B) = Z(R)B$. Let $f \in Z(B)$, then there exists a nonzero $r \in R$ such that $f.r = 0$. Therefore we have $c(f) \subseteq Z(R)$ and since $Z(R)$ is an ideal of $R$, $f \in Z(R)B$. Now let $f \in Z(R)B$, then $c(f) \subseteq Z(R)$. Since $R$ has Property (A), $c(f)$ has a nonzero annihilator and this means that $f$ is a zero-divisor in $B$. So we have already shown that $Z(B)$ is an ideal of $B$ and therefore $B$ is primal. Finally we prove that $B$ has Property (A). Assume that $J=(f_1,f_2, \ldots, f_n) \subseteq Z(B)$. Therefore $c(f_1), c(f_2),\ldots, c(f_n) \subseteq Z(R)$. But $Z(R)$ is an ideal of $R$ and $c(f_i)$ is a finitely generated ideal of $R$ for any $1\leq i \leq n$, so $I=c(f_1)+c(f_2)+\cdots + c(f_n) \subseteq Z(R)$ is a finitely generated ideal of $R$ and there exists a nonzero $s \in R$ such that $sI=0$. This causes $sJ=0$ and $J$ has a nonzero annihilator in $B$.

$(3)$: We just need to prove that if $B$ is primal, then $R$ has Property (A). For doing that let $I \subseteq Z(R)$ be a finitely generated ideal of $R$. Since the content function is onto, there exists an $f\in B$ such that $I = c(f)$. Since $c(f) \subseteq Z(R)$, $f\in Z(B)$. According to McCoy's property for content algebras, we have $f\cdot r =0$ for some nonzero $r \in R$ and this means $I=c(f)$ has a nonzero annihilator and the proof is complete.
\end{proof}

\section{Zero-divisor graph of content algebras}

Let $R$ be a commutative ring with identity and proper zero-divisors, where by a proper zero-divisor we mean a zero-divisor different from zero. We let $Z(R)^*$ denote the set of proper zero-divisors of $R$. We consider the graph $\Gamma(R)$, called zero-divisor graph of $R$, whose vertices are the elements of $Z(R)^*$ and edges are those pairs of distinct proper zero-divisors $\{a,b\}$ such that $ab=0$.

Recall that a graph is said to be connected if for each pair of distinct vertices $v$ and $w$, there is a finite sequence of distinct vertices $v=v_1, v_2 ,\ldots , v_n = w$ such that each pair $\{v_i , v_{i+1} \}$ is an edge. Such a sequence is said to be a path and the distance, $d(v, w)$, between connected vertices $v$ and $w$ is the length of the shortest path connecting them. The diameter of a connected graph $G$ is the supremum of the distances between vertices and is denoted by $\diam(G)$. In [AL], zero-divisor graphs were studied and among many things, it was proved that any zero-divisor graph, $\Gamma(R)$, is connected with $0 \leq \diam(\Gamma(R)) \leq 3$ [AL, Theorem 2.3]. Note that the diameter is $0$ if the graph consists of a single vertex and a connected graph with more than one vertex has diameter $1$ if and only if it is complete; i.e., each pair of distinct vertices forms an edge.

In this section, we examine the preservation of diameter of zero-divisor graph under content extensions. What we do is the generalization of what it has been done for polynomial rings in [ACS] and [L]. The following lemmas are straightforward, but we bring them only for the sake of reference.

\begin{lemma}
 Let $R$ be a ring and $B$ a content $R$-algebra. Then the following statements hold:

\begin{enumerate}

\item $\Nil(B)=\Nil(R)B$, where by $\Nil(R)$ we mean all nilpotent elements of $R$.
\item $Z(R) \subseteq Z(B) \subseteq Z(R)B$.
\item $Z(R)^n=(0)$ iff $Z(B)^n=(0)$ for all $n\geq1$.

\end{enumerate}

\end{lemma}

\begin{proof}
 It is well-known that the set of all nilpotent elements of a ring is equal to the intersection of all its minimal primes. For proving (1), use Theorem 5 and [OR, 1.2, p. 51]. The statements (2) is obvious by definition of content modules and [OR, 6.1]. For (3), suppose that $Z(R)^n=0$. Choose $f_1,f_2,\ldots,f_n \in Z(B)$. Therefore by McCoy's property for content algebras, $c(f_1),c(f_2),\ldots,c(f_n) \subseteq Z(R)$. But by [R, Proposition 1.1] $c(f_1f_2 \cdots f_n)\subseteq c(f_1)c(f_2)\cdots c(f_n) \subseteq Z(R)^n=(0)$. Hence $f_1f_2 \cdots f_n=0$.
\end{proof}

\begin{lemma}
 Let $R$ be a ring and $B$ a content $R$-algebra. Then $\diam (\Gamma(R)) \leq \diam (\Gamma(B))$.
\end{lemma}

\begin{proof}
 Note that the defining homomorphism of $R$ into $B$ is injective and therefore we can suppose $R$ to be a subring of $B$ [OR, Remark 6.1(b)] and so $Z(R)^* \subseteq Z(B)^*$. It is obvious that if $\diam (\Gamma(R))=0,1, \text{~or~} 2$, then no path in $\Gamma(R)$ can have a shortcut in $\Gamma(B)$. Now let $\diam (\Gamma(R))=3$ and $a - b - c - d$ be the path in $\Gamma(R)$ with $a,b,c,d \in Z(R)^*$ without having any shortcut. Our claim is that neither is there a shortcut for this path in $Z(B)^*$. On the contrary suppose that there is an $h \in Z(B)^*$ such that $a - h - d$ is a path in $ Z(B)^*$. Then $a\cdot c(h)=d \cdot c(h)=(0)$. Since $h\neq 0$, there exists a nonzero element $r \in c(h)$ such that $ar=rd=0$ and this means that $a - r - d$ is a shortcut in $\Gamma(R)$, a contradiction. Therefore $\diam (\Gamma(B))=3$. This means that in any case the inequality $\diam (\Gamma(R)) \leq \diam (\Gamma(B))$ holds.
\end{proof}

Recall that $\diam (\Gamma(R))=0$ iff $R$ is isomorphic to either $\mathbb Z_4$ or $\mathbb Z_2[y]/(y^2)$ [AL, Example 2.1] and $\diam (\Gamma(R))=1$ iff $xy=0$ for each pair of distinct zero-divisors of $R$ and $R$ has at least two proper zero-divisors [AL, Theorem 2.8]. Also from [L, Theorem 2.6(3)], we know that $\diam (\Gamma(R))=2$ iff either (i) $R$ is reduced with exactly two minimal primes and at least three proper zero-divisors or (ii) $Z(R)$ is an ideal whose square is not (0) and each pair of distinct zero-divisors has a nonzero annihilator. These facts help us to examine the preservation of diameter of zero-divisor graph under content extensions for the cases $\diam (\Gamma(R))=0, 1$.

\begin{theorem}
Let $R$ be a ring and $B$ be a content $R$-algebra and $B\ncong R$. Then the following statements hold:

\begin{enumerate}

\item $\diam (\Gamma(R))=0$ and $\diam (\Gamma(B))=1$ iff either $R \cong \mathbb Z_4$ or $R \cong \mathbb Z_2[y]/(y^2)$.

\item $\diam (\Gamma(R))= \diam (\Gamma(B))=1$ iff $R$ is a nonreduced ring with more than one proper zero-divisor and $Z(R)^2 =0$.

\item $\diam (\Gamma(R))=1$ and $\diam (\Gamma(B))=2$ iff $R \cong \mathbb Z_2 \times \mathbb Z_2$.

\end{enumerate}

\end{theorem}

\begin{proof}

Let $B$ be a content $R$-algebra such that $B \ncong R$. For (1), we just need to prove that if $\diam (\Gamma(R))=0$, then $\diam (\Gamma(B))=1$. It is obvious that $R$ is isomorphic to either $\mathbb Z_4$ or $\mathbb Z_2[y]/(y^2)$ [AL, Example 2.1]. But $Z(\mathbb Z_4) = (2)$ and $Z(\mathbb Z_2[y]/(y^2))= \{by+(y^2): b \in \mathbb Z_2 \}=(y)$, so in any case $Z(B)= Z(R)B$ by Theorem 22(2). It is easy to check that $fg=0$ for any distinct pair of zero-divisors $f,g$ in $B$ and $B$ has at least two proper zero-divisors. So according to [L, Theorem 2.6(2)], $\diam (\Gamma(B))=1$.

(2) and (3): If $R$ is a nonreduced ring with more than one proper zero-divisor and $Z(R)^2=(0)$ then $R\ncong \mathbb Z_2 \times \mathbb Z_2$ and $ab=0$ for all $a,b\in Z(R)$ [AL, Theorem 2.8]. This means that $\diam (\Gamma(R))=1$. Let $f,g \in Z(B)^*$. According to McCoy's property for content algebras, there exist nonzero $r,s \in R$ such that $c(f)\cdot r=c(g)\cdot s=0$. This implies that $c(f) \subseteq Z(R)$ and $c(g) \subseteq Z(R)$ and therefore $c(fg) \subseteq c(f)c(g) = (0)$. But $B$ has at least two proper zero-divisors, since $Z(R)^* \subseteq Z(B)^*$. Hence $\diam (\Gamma(B))=1$.

Now let $R\cong \mathbb Z_2 \times \mathbb Z_2$. Then $R$ is a reduced ring with exactly two minimal prime ideals, $\textbf{p}$ and $\textbf{q}$, where $\textbf{p}= ((1,0))$ and $\textbf{q}=((0,1))$ and according to Lemma 23 and Corollary 6, $B$ is a reduced ring with exactly two minimal prime ideals, $\textbf{p}B$ and $\textbf{q}B$. It is obvious that $B$ has at least two proper zero-divisors. If $B$ has exactly two proper zero-divisors, then $Z(B)^*=\{(1,0),(0,1)\} $ and therefore according to [AL, Theorem 2.8], $B\cong \mathbb Z_2 \times \mathbb Z_2$. Therefore $B$ has at least three proper zero-divisors and according to [L, Theorem 2.6(3)], we have $\diam (\Gamma(B))=2$. By this discussion, it is, then, obvious that $\diam (\Gamma(R))= \diam (\Gamma(B))=1$ implies $R$ is a nonreduced ring with more than one proper zero-divisor and $Z(R)^2 =0$.

Now let $\diam (\Gamma(R))=1$ and $\diam (\Gamma(B))=2$. If $Z(R)^2=(0)$, then $Z(B)^2=0$ and $\diam (\Gamma(B))=1$, therefore $R \cong \mathbb Z_2 \times \mathbb Z_2$ by [AL, Theorem 2.8].
\end{proof}

Now we examine the preservation of diameter of zero-divisor graph under content extensions with $\diam (\Gamma(R))=2$.

\begin{theorem}
 Let $B$ be a content $R$-algebra such that the content function $c: B \longrightarrow \FId(R)$ is onto. Let $R$ has at least three proper zero-divisors and  $B\ncong R$. Then the following statements hold:

\begin{enumerate}

 \item $\diam (\Gamma(R))= \diam (\Gamma(B))=2$ iff either (i) $R$ is a reduced ring with exactly two minimal prime ideals and $R$ has more than two proper zero-divisors, or (ii) $R$ is a primal ring with $Z(R)^2 \neq (0)$ and $R$ has Property (A).

\item $\diam (\Gamma(R))=2$ and $\diam (\Gamma(B))=3$ iff $Z(R)$ is an ideal of $R$ and $R$ does not have Property (A) but each pair of proper zero-divisors of $R$ has a nonzero annihilator.

\end{enumerate}

\end{theorem}

\begin{proof}
 (1): If $R$ is a reduced ring with exactly two minimal prime ideals and $R$ has more than two proper zero-divisors, then according to Lemma 23 and Corollary 6, $B$ is a reduced ring with exactly two minimal prime ideals and obviously $B$ has more than two proper zero-divisors and therefore according to [L, Theorem 2.6(3)], $\diam (\Gamma(R))= \diam (\Gamma(B))=2$. If $R$ is a primal ring with $Z(R)^2 \neq (0)$ and $R$ has Property (A), then by Theorem 22(2), $B$ is primal and has Property (A). Also obviously $Z(B)^2\neq 0$. So according to [L, Theorem 2.6(3)], $\diam (\Gamma(R))= \diam (\Gamma(B))=2$. Now let $\diam (\Gamma(R))= \diam (\Gamma(B))=2$. If $R$ is a reduced ring with exactly two minimal prime ideals and $R$ has more than two proper zero-divisors, then we are done, otherwise, $Z(B)$ is an ideal whose square is not (0) and each pair of distinct zero-divisors has a nonzero annihilator. Since $Z(B)$ is primal, then $Z(R)$ is an ideal of $R$ and $R$ has Property (A) by Theorem 22(3). But $Z(B)^2\neq0$ implies that $Z(R)^2\neq0$ and the proof is complete.

(2) Assume that $Z(R)$ is an ideal of $R$ and $R$ does not have Property (A) but each pair of proper zero-divisors of $R$ have a nonzero annihilator. It is obvious that $\diam (\Gamma(R))=2$. Our claim is that $\diam (\Gamma(B))=3$. On the contrary, let $\diam (\Gamma(B))=2$. According to [L, Theorem 2.6(3)] and [H, Corollary 2.4], either $\zd(R)=2$ or $\zd(R)=1$. But the content function $c: B \longrightarrow \FId(R)$ is onto and in both cases, by Theorem 18(3), $R$ has Property (A), a contradiction. Therefore $\diam(\Gamma(B))=3$. Now let $\diam (\Gamma(R))=2$ and $\diam (\Gamma(B))=3$, then according to [L, Theorem 2.6(3)] and Theorem 26(1), $Z(R)$ is an ideal of $R$ and each pair of proper zero-divisors of $R$ has a nonzero annihilator. Obviously $R$ does not have Property (A), otherwise $\diam(\Gamma(B))=2$ and the proof is complete.
\end{proof}

Note that the two recent theorems are the generalization of [L, Theorem 3.6]. Consider that in the last theorem, we assume the content function $c: B \longrightarrow \FId(R)$ to be onto. In the following we state a theorem similar to [ACS, Proposition 5], without assuming the content function $c: B \longrightarrow \FId(R)$ to be onto.

\begin{theorem}
 Let $B$ be a content $R$-algebra and $Z(R)^n=(0)$, while $Z(R)^{n-1}\neq(0)$ for some $n\geqq2$. Then the following statements hold:

\begin{enumerate}

\item If $n=2$, then $\diam(\Gamma(R))=\diam(\Gamma(B))=1$.

\item If $n>2$, then $\diam(\Gamma(R))=\diam(\Gamma(B))=2$.

\end{enumerate}

\end{theorem}

\begin{proof}
 (1) holds by [AL, Theorem 2.8] and Theorem 25(2).

(2): By assumption $Z(R)^n=(0)$ and $Z(R)^2\neq(0)$. Therefore $\Gamma(R)$ is not a complete graph and so there exist distinct $a,b \in Z(R)^*$ such that $ab\neq0$. Since $Z(R)^{n-1}\neq(0)$, there exist $c_1,c_2,\ldots,c_{n-1} \in Z(R)$ such that $c=c_1c_2 \cdots c_{n-1}\neq0$. So $c\neq a,b$ and $ca=cb=0$. Hence $\diam(\Gamma(R))=2$. On the other hand $Z(R)^{n-1}\neq(0)$ causes $Z(B)^{n-1}\neq(0)$. By Lemma 23, $Z(B)^n=(0)$. This means that $\diam(\Gamma(B))=2$ and the proof is complete.
\end{proof}

\section{Acknowledgment} The author wishes to thank Prof. Winfried Bruns and the referee for their useful advice.

\end{document}